\documentclass[11pt]{amsart}%fleqn
\usepackage{fullpage}

\usepackage{amsmath}
\usepackage{amssymb}
\usepackage{amsfonts}
\usepackage{graphicx}
\usepackage{amsthm}
\usepackage{enumerate}
\usepackage{lscape}
\usepackage{dsfont}
\usepackage{color}
\usepackage{mathtools}
\usepackage[dvipsnames]{xcolor}
\usepackage{hyperref}
\usepackage{appendix}

\newcounter{example}

\usepackage[utf8]{inputenc}

\newcommand{\C}{\mathds C}
\newcommand{\R}{\mathds R}

\newcommand{\CP}{\mathds C {\rm P}}

\usepackage{apptools}

\newtheorem{theorem}{Theorem}

\newtheorem{lemma}{Lemma}
\newtheorem{proposition}{Proposition}

\newtheorem{remark}{Remark}

\makeatletter
\@namedef{subjclassname@2020}{%
  \textup{2020} Mathematics Subject Classification}
\makeatother

\numberwithin{equation}{section}

\def\subsubsection{\@startsection{subsubsection}{3}%
\z@{.5\linespacing\@plus.7\linespacing}{-.5em}%
{\normalfont\bfseries}}
\makeatother

\title{K\"ahler geometry of scalar flat metrics on Line bundles over polarized K\"ahler--Einstein manifolds}

\author{Simone Cristofori}
\address{(Simone Cristofori) Dipartimento di Scienze Matematiche, Fisiche e Informatiche \\
         Universit\`a di Parma (Italy)}
         
        \email{simone.cristofori@unipr.it}
\author{Michela Zedda}
\address{(Michela Zedda) Dipartimento di Scienze Matematiche, Fisiche e Informatiche \\
         Universit\`a di Parma (Italy)}
\email{michela.zedda@unipr.it}
\date{\today}

\thanks{
The second named author has been supported by the project Prin 2022 – Real and Complex Manifolds: Geometry and Holomorphic Dynamics – Italy. Both the authors were supported by INdAM GNSAGA - Gruppo Nazionale per le Strutture Algebriche, Geometriche e le loro Applicazioni. 
} 

\begin{document}

\maketitle

\begin{abstract}
In view of a better understanding of the geometry of scalar flat K\"ahler metrics,
this paper studies two families of scalar flat K\"ahler metrics constructed in \cite{hwangsinger} by A. D. Hwang and M. A. Singer on $\mathds C^{n+1}$ and on $\mathcal O(-k)$. For the metrics in both the families, we prove the existence of an asymptotic expansion for their $\epsilon$-functions and we show that they can be approximated by a sequence of projectively induced K\"ahler metrics. Further, we show that the metrics on $\mathds C^{n+1}$ are not projectively induced, and that the Burns--Simanca metric is characterized among the scalar flat metrics on $\mathcal O(-k)$ to be the only projectively induced one as well as the only one whose second coefficient in the asymptotic expansion of the $\epsilon$-function vanishes.
\end{abstract}
\tableofcontents

\section{Introduction and statement of the main result}
An important open problem in K\"ahler geometry consists in characterizing projectively induced metrics in view of the properties of their curvatures. 
A K\"ahler metric $g$ on a complex manifold $M$ is said to be projectively induced if there exists a {\em local} K\"ahler immersion into the complex projective space $\mathds C{\rm P}^N$, that is if for any $p\in M$ there exists an open set $U\subset M$, $p\in U$, and a holomorphic function $f\!:U\rightarrow \mathds C{\rm P}^N$, such that $f^*g_{FS}=g$.
Here we denote by $g_{FS}$ the Fubini--Study metric, i.e. if $[Z_0:\dots:Z_N]$ are homogeneous coordinates on $\CP^N$ and $(z_1,\dots, z_N)$ are affine coordinates on $U_0=\{Z_0\neq 0\}$, $g_{FS}$ is described on $U_0$ by the K\"ahler potential $\log(1+|z_1|^2+\dots+|z_N|^2)$. Observe that we allow $N$ to be infinite, where $\CP^\infty$ is the quotient of $l^2(\C)\setminus\{0\}$ by the usual equivalent relation.

Many examples of projectively induced metrics can be constructed by taking the pull-back of the Fubini--Study metric on holomorphic submanifolds of $\CP^N$, although, it is more difficult to find projectively induced metrics with prescribed curvature. For example, D. Hulin in \cite{hulin} proved that the scalar curvature of a compact K\"ahler--Einstein manifold K\"ahler immersed into $\CP^N$, is forced to be positive. Observe that if a compact manifold admits a K\"ahler immersion in $\CP^\infty$ then it is also a K\"ahler submanifold of $\CP^N$ for some finite $N$, as the immersion is given by a basis of the space of global holomorphic sections of a suitable holomorphic line bundle, that when the manifold is compact is always finite dimensional. Although, this holds true only for {\em global} K\"ahler immersions, in fact the flat torus is an example of compact manifold that is {\em locally} projectively induced in $\CP^\infty$ but does not admit any K\"ahler immersion in $\CP^N$ for finite $N$, as follows by Calabi's rigidity Theorem in \cite[Th. 9]{calabi} (see also \cite{loizeddabook} for an overview of Calabi's work).
Recently in \cite{ALL}, C. Arezzo, C. Li and A. Loi proved that there are not Ricci--flat submanifolds of $\CP^N$ with $N<\infty$. It is still an open question if there exists a Ricci--flat (nonflat) K\"ahler submanifold of $\CP^\infty$.  
It is important to emphasize that when the ambient space is taken to be infinite dimensional the situation could be much different, for example in \cite{loizeddaMAnn} K\"ahler-Einstein submanifolds of $\CP^\infty$ with negative scalar curvature are given.
In  \cite{loisaliszuddas} A. Loi, F. Salis and F. Zuddas conjectured that the flat metric is the only example of projectively induced Ricci--flat metric and they validate the conjecture when the metric is radial and the immersion is {\em stable} (see also \cite{loizeddazuddas,loizeddazuddas2,zedda} for other results in the same context).

A very little is known for constant scalar curvature K\"ahler metrics. In the finite dimensional context, it is conjectured by A. Loi, F. Salis, F. Zuddas in \cite{loisaliszuddas3} that  the only projectively induced constant scalar curvature K\"ahler metrics lie on flag manifolds (actually their conjecture includes also extremal K\"ahler metrics). 
The Burns--Simanca metric on the blow--up of $\mathds C^2$ at one point is an example of scalar flat (nonflat) complete projectively induced K\"ahler metric, as shown by F. Cannas--Aghedu and A. Loi in \cite{agheduloi}. The Burns--Simanca metric actually satisfies a stronger assumption than to be projectively induced, namely it admits a regular quantization. 

A geometric quantization $(L,h)$ of a $n$-dimensional K\"ahler manifold $(M, \omega)$ consists of an hermitian holomorphic line bundle $L$ over $M$ such that the first Chern class of $L$ is represented by $\omega$ and its curvature ${\rm Ric}(h):=-i\partial \bar \partial \log h$ satisfies ${\rm Ric}(h)=\omega$. Let $\mathcal H$ be the space of global holomorphic sections of $L$ and denote by $\langle\cdot,\cdot\rangle_{h}$ the scalar product:
$$
\langle s,s\rangle_{h}:=\int_Mh(s(x),s(x))\frac{\omega^n}{n!}.
$$
When $\mathcal H\neq \{0\}$ (condition that is always satisfied when $M$ is compact) we can take an orthonormal basis $\{s_j\}_{j=0,\dots,d}$ of $\mathcal H$, and define a function on $M$ by:
\begin{equation}\label{epsilondef}
\epsilon_{g}(x):=\sum_{j=0}^{d}h(s_j(x),s_j(x)).
\end{equation}
In literature this $\epsilon$-function was first introduced under the name of $\eta$-{\em function} by J. Rawnsley in \cite{rawnsley}, later renamed as $\theta$-{\em function} in \cite{cgr1} followed by the {\em distortion function } of G. R. Kempf \cite{ke} and S. Ji
\cite{ji}, for the special case of Abelian varieties and of S. Zhang \cite{zha} for complex projective varieties.
The geometric quantization $(L,h)$ of $(M,\omega)$ is said to be {\em regular} if $\epsilon_{\alpha g}$ is constant for all large enough $\alpha\in \mathds Z^+$ (when $M$ is noncompact, $k$ is not necessarily an integer). Observe that in this case one considers the geometric quantizations given by $(L^\alpha,h_\alpha)$ such that ${\rm Ric}(h_\alpha)=\alpha\omega$. We also say that the metric is regular. Regular metrics enjoy the properties of being projectively induced K\"ahler metrics of constant scalar curvature.
More precisely, for large enough $\alpha$ one can construct a holomorphic map $F_\alpha\!:M\rightarrow \CP^{d_\alpha}$, ($d_\alpha\leq+\infty$), called the \emph{coherent states map}, by:
\[
F_\alpha:M\to\CP^{d_\alpha}\quad;\quad x\mapsto[s_0(x):\dots:s_{d_\alpha}(x)].
\]
which satisfies (see e.g. \cite{arezzoloi}):
\begin{equation}\label{csm}
    F_\alpha^*(\omega_{FS})=\alpha\omega+\frac{i}{2}\partial\overline\partial\log\epsilon_{\alpha g}.
\end{equation}
In particular one has that when $\epsilon_{\alpha g}$ is constant, $F_\alpha$ is a holomorphic and isometric immersion. 

Further, in view of Zelditch work \cite{zelditch}, when $M$ is compact the function $\epsilon_{\alpha g}$ admits an asymptotic expansion (the so called {\em Tian-Yau-Catlin-Zelditch expansion}):
$$
\epsilon_{\alpha g}(x)\sim\sum_{j=0}^\infty a_j(x) \alpha^{n-j},
$$
where $a_0(x)\equiv 1$ and the $a_j(x)$, $j=1,2,\dots$ are smooth functions on $M$ depending on the curvature and on its covariant derivatives at $x$ of $g$.
For this asymptotic expansion it is meant that, for every integers $l,r$ and every compact $K\subseteq M$,
\begin{equation}\label{asympt}
\left|\left|\epsilon_{\alpha g}(x)-\sum_{j=0}^l a_j(x) \alpha^{n-j} \right|\right|_{C^r}\le\frac{C(l,r,K)}{\alpha^{l+1}},
\end{equation}
for some constant $C(l,r,K)>0$.
In particular, Z. Lu \cite{lu} computed the first three coefficients, and the first two reads:
\begin{equation}
\begin{cases} \label{coeffespan} a_1=\frac12\sigma_g\\ a_2=\frac13\Delta\sigma_g+\frac1{24}\left(|R_g|^2-4|{\rm Ric}_g|^2+3\sigma_g^2\right),\end{cases}
\end{equation}
where $\sigma_g$, ${\rm Ric}_g$ and $R_g$ denote respectively the scalar curvature, the Ricci tensor and the curvature tensor of $g$, and the norms are taken with respect to $g$. When $M$ is noncompact the existence of such an expansion (known as {\em Engli\v{s} expansion}) is not guaranteed and only partial results are given (see Section \ref{sec: englisexp} for details). Further, in \cite{engliscoeff} M. Engli\v{s} computed the $a_j$'s coefficients obtaining the same results as Lu. All the $a_j$'s coefficients of regular metrics are constant. The Burns-Simanca metric shares with the flat metric the property of presenting all the coefficients $a_j$'s equal to zero \cite{agheduloi}.

Even if the metric is not regular, the existence of an asymptotic expansion for the $\epsilon$ function has important geometric consequences. In particular, it turns out that the coherent states map via \eqref{csm} allows to approximate a K\"ahler metric $g$ with projectively induced ones (see Lemma \ref{convergence} in Section \ref{sec: englisexp}).

In this paper we study families of scalar flat metrics constructed via Calabi ansatz on the total space of a hermitian line bundle over K\"ahler--Einstein manifolds by Andrew D. Hwang and Michael A. Singer in \cite{hwangsinger}. The necessary hypotheses for the existence of scalar flat metrics include the so called sigma constancy, condition that is automatically satisfied by polarized K\"ahler manifolds, which is the case we are interested in. In particular we consider the following families:
\begin{enumerate}
    \item[(A)] the $1$-parameter family of nontrivial scalar flat K\"ahler metrics $g_\beta$ on $\C^{n+1}$, $\beta<0$ (described in Section \ref{flatcase});
\item[(B)] the scalar flat metrics $g_k$ on $\mathcal O(-k)$ for integers $k>0$ (described in Section \ref{projectivecase}).
\end{enumerate}
Observe that the metrics $g_k$ in (B) reduce to the Burns--Simanca metric for $k=1$, and to the Ricci--flat Eguchi--Hanson metric for $k=2$. 

The first result of this paper is the following:
\begin{theorem}\label{main1}
 Let $g_\beta$ be the K\"ahler metric on $\C^{n+1}$ arising from Hwang--Singer construction. Then $c g_\beta$ is not projectively induced for any value of $c>0$ and $\beta<0$, but it can be approximated by a sequence of projectively induced metrics. 
 \end{theorem}

Our second result characterizes the Burns--Simanca metric among the Hwang--Singer family $g_k$ on $\mathcal O(-k)$. More precisely we prove the following:
\begin{theorem}\label{main2}
 Let $g_k$ be the K\"ahler metric arising from Hwang--Singer construction on $\mathcal O(-k)$. Then $g_k$ is projectively induced if and only if its second coefficient vanishes identically, that is if and only if it is the Burns--Simanca metric on the blow-up of $\C^2$ at one point. Moreover $g_k$ can be approximated by a sequence of projectively induced metrics.
\end{theorem}

The paper is organized as follows. In Section \ref{sec: momentumconstruction} we recall what we need about Hwang-Singer construction restricted to polarized K\"ahler--Einstein manifolds. In Section \ref{sec: HS diastasis} we give an overview of Calabi's criterion, deriving a necessary condition for the Hwang-Singer metrics to be projectively induced. Sections  \ref{flatcase} and \ref{projectivecase} are devoted respectively to the description of Hwang-Singer metrics on $\C^{n+1}$ and $\mathcal{O}(-k)$. Section \ref{sec: englisexp} contains the existence results for the $\epsilon$-function associated to the Hwang-Singer metrics on $\C^{n+1}$ and $\mathcal{O}(-k)$, and for its asymptotic expansion, and the proofs of theorems \ref{main1} and \ref{main2}. Finally, the appendix includes some computations regarding the $a_2$ coefficient. \\

{\small {\em Acknowledgements:} The authors are grateful to Andrea Loi and Roberto Mossa for their valuable suggestions which have contributed to the improvement of this paper.}

\section{Momentum construction}\label{sec: momentumconstruction}

A technique to produce complete K\"ahler metrics with good curvature properties is known as \emph{Calabi ansatz}, firstly introduced by E. Calabi in \cite{Cal} and later adopted by several authors. Andrew D. Hwang and Michael A. Singer generalized this construction on the total space of an hermitian holomorphic line bundle $\pi:L\to M$ with ``$\sigma$-\emph{constant curvature}'' over a K\"ahler manifold $(M,g_M)$. In this section we summarize Hwang--Singer construction, restricting our attention to the case of polarized manifolds, where these hypothesis are automatically satisfied.

Let $\pi:(L,h)\to (M,\omega_M)$ be a polarized hermitian holomorphic line bundle with curvature form $\gamma=-i\partial\overline\partial\log h\in\Omega^2(M)$ such that $\gamma=\beta\omega_M$ over a  K\"ahler--Einstein manifold of complex dimension $n$, that is $\rho_M=\lambda\omega_M$, where $\rho_M$ is the Ricci form associated to $g_M$. 
This method, also known as \emph{momentum construction}, gives rise to \emph{bundle-adapted metrics} on $L$, that is K\"ahler metrics $g_{\varphi,\beta}$ whose K\"ahler form arises from the Calabi ansatz
\[
\omega_{\varphi,\beta}=\pi^*\omega_M+2i\partial\overline\partial f(t),
\]
where $t$ is the logarithm of the norm function defined by $h$ and $f:(-\infty,+\infty)\to[0,+\infty)$ is an increasing and strictly convex function of one real variable which makes $\omega_{\varphi,\beta}$ positive definite.

 In a coordinate chart $U\subset M$ over which $L$ is trivial, i.e. $\pi^{-1}(U)\cong U\times\C$, there exists a local coordinate system $\tilde z=(\xi,z)=(\xi,z^1,\dots,z^n)$ for $L$ where $\xi=\rho e^{i\theta}$ is a fibre coordinate and $z=(z^1,\dots,z^n)$ are pullbacks of coordinates on $M$, i.e., if $q\in M$ is a point with coordinates $z$, then every point in the fiber $\pi^{-1}(q)$ can be described by coordinates $\tilde z$. In such a chart there is a smooth positive function $h:U\subset M\to\R$ such that
\[
t:=\log||\tilde z||=\frac{1}{2}\log\left(|\xi|^2h(z)\right).
\]

As explained in \cite{hwangsinger}, to simplify the construction of scalar flat K\"ahler metrics on $L$, it is advantageous to change coordinates.  Setting
\[
\tau=f'(t)\quad,\quad \varphi(\tau)=f''(t)\;,
\]
so that $f$ satisfies the differential equation
\[
\begin{cases}
f''(t)=\varphi(\tau)\\
f'(0)=\mu_0>0
\end{cases}\;,
\]
the K\"ahler metric $\omega_{\varphi,\beta}$ reads as
\begin{equation}\label{omegaphibeta}
\omega_{\varphi,\beta}=\pi^*\omega_M-\tau\pi^*\gamma+\frac{1}{\varphi}\;d\tau\wedge d^c \tau,
\end{equation}
and along the fibre $L_x$ over $x\in M$, restricts to 
\[
\omega_{\varphi,\beta}|_{\text{fibre}}=\frac{\varphi(\tau)}{|\xi|^2}\;d\xi\wedge d\overline{\xi}\;.
\]
The explicit expression for the profile function $\varphi$ for polarized metrics is:
\begin{equation}\label{profile}
   % \begin{split}
        \varphi(\tau) 
                = \frac{2}{(1-\beta\tau)^n}\bigg(\tau+\frac{\lambda\big((1-\beta\tau)^{n+1}-(1-\beta\tau)+\beta n\tau\big)}{\beta^2(n+1)}\bigg).
   % \end{split}
\end{equation}
Observe that the factor $(1-\beta\tau)^n$ arises as the determinant of the endomorphism ${\rm Id}-\tau {\rm B}$, since ${\rm B}:=\omega_M^{-1}\gamma=\beta\, {\rm Id}$ for $\gamma=\beta\omega_M$.  

\begin{remark}\label{finfty}\rm
    The function $f':(-\infty,+\infty)\to I:=(0,+\infty)$ is an increasing and surjective function (see Prop. 1.4. in \cite{hwangsinger}). In particular
    \[
    \lim_{t\to-\infty}f'(t)=0.
    \]
\end{remark}

\begin{remark}\rm \label{rem: derivatef}
    The derivatives of the function $f(t)$ are expressed recursively in the variable $\tau$ as
    \[
    f^{(n)}(t)=\varphi(\tau)(f^{(n-1)}(t))
    \]
    for $n\ge3$.
    In particular we have
    \begin{equation}\label{derivatef}
    \begin{split}
        f'''(t)&=\varphi(\tau)\varphi'(\tau),\\
         f^{(iv)}(t)&=\varphi(\tau)(\varphi(\tau)\varphi''(\tau)+(\varphi'(\tau))^2).
    \end{split}    
    \end{equation}
  
\end{remark}

\begin{proposition}\label{cf}
    Let $c>0$ be a positive real number. If $f$ is a solution for the {\rm ODE} $y''=\varphi(y')$ with $\varphi$ given by \eqref{profile},
    then $\hat{f}:=cf$ is a solution to $y''= \hat\varphi(y')$, where we denote with $\hat\varphi$ the profile function with parameters $\hat{\beta}=\frac{\beta}{c}$ and $\hat{\lambda}=\frac{\lambda}{c}$. 
\end{proposition}
\begin{proof}
    It follows by noticing that 
    \begin{equation*}
    \begin{split}
    \varphi(y')&=\frac{2}{(1-\frac{\beta}{c} (cy)')^n}\bigg(\frac{1}{c}(cy)'+\frac{1}{c^2}\frac{\lambda\big((1-\frac{\beta}{c} (cy)')^{n+1}-(1-\frac{\beta}{c} (cy)')+n\frac{\beta}{c} (cy)'\big)}{\frac{\beta^2}{c^2}(n+1)}\bigg)\\
    &=\frac{1}{c}\frac{2}{(1-\hat{\beta} (cy)')^n}\bigg(y'+\frac{\hat{\lambda}\big((1-\hat{\beta} (cy)')^{n+1}-(1-\hat{\beta} (cy)')+\hat{\beta} n (cy)'\big)}{\hat{\beta}^2(n+1)}\bigg).
    \end{split}
    \end{equation*}
Thus
\[
c\varphi(y')=\frac{2}{(1-\hat{\beta} (cy)')^n}\bigg(y'+\frac{\hat{\lambda}\big((1-\hat{\beta} (cy)')^{n+1}-(1-\hat{\beta} (cy)')+\hat{\beta} n (cy)'\big)}{\hat{\beta}^2(n+1)}\bigg).
\]
\end{proof}

In this setting \cite[Theorem B]{hwangsinger} by A. D. Hwang and M. A. Singer, reads:

\begin{theorem}
Let $\pi:(L,h)\to (M,\omega_M)$ be a polarized hermitian holomorphic line bundle over a complete K\"ahler--Einstein manifold $(M,\omega_M)$ with $\gamma<0$, such that $\gamma=\beta\omega_M$, with $\beta<0$. Then the metric $g_{\varphi,\beta}$ on the total space of $L$ is a complete scalar flat K\"ahler metric.
Moreover, the metric $g_{\varphi,\beta}$ is Ricci--flat if and only if $\rho_M=-\gamma$.
\end{theorem}

\begin{remark}\rm   \label{prodmetrics}
For $\gamma=0$ we have local product metrics since they are bundle-adapted metrics on flat-bundles, see (\cite{hwangsinger}, Remark 1.6).
\end{remark}

\section{Calabi's Criterion applied to $g_{\varphi,\beta}$}\label{calabicriterion}\label{sec: HS diastasis}
In this section we recall what we need on Calabi's criterion for projectively induced metrics and compute the diastasis function for our metrics.

Let $(\CP^N,g_{\text{FS}})$ be the complex projective space of dimension $N\le\infty$ endowed with the Fubini-Study metric. Let $[Z_0:\dots:Z_N]$ be homogeneous coordinates and $(z_1,\dots,z_N)$ the respective affine coordinates on the coordinate charts $U_j=\{Z_j\neq0\}$ defined by $z_k:=\frac{Z_k}{Z_j}$. A K\"ahler potential for $g_{\text{FS}}$ on $U_0$ is 
\[
\phi_{\text{FS}}(z)=\log\biggl(1+\sum_{j=1}^N|z_j|^2\biggl).
\]

In \cite{calabi} E. Calabi gives a criterion for a real analytic K\"ahler manifold $(M,g)$ to admit a holomorphic and isometric (from now on {\em K\"ahler}) immersion into a complex space form in terms of the {\em diastasis function} associated to the metric $g$. Here we consider only the case when the ambient space is the complex projective space $\CP^N$, that is when the metric is projectively induced, which is the one we deal with. Observe that it is not restrictive to assume the manifold to be real analytic, since a metric induced by the pull-back through a holomorphic map of the real analytic Fubini-Study metric, is forced itself to be real analytic. The diastasis function can be viewed as a particular K\"ahler potential defined as follows. Fix a coordinates system $(z_1,\dots,z_n)$ on a chart $U\subset M$ and let $\phi:U\to\R$ be a K\"ahler potential for $g$ on $U$.
 By duplicating the variables $z$ and $\overline z$, the K\"ahler potential $\phi$ on $U$ can be complex analytically extended to a function $\tilde\phi:W\to\R$ on a neighborhood $W$ of the diagonal in $U\times\overline U$. The diastasis function is defined by
\begin{equation}\label{diastasis}
D(z,w):=\tilde\phi(z,\overline z)+\tilde\phi(w,\overline w)-\tilde\phi(z,\overline w)-\tilde\phi(w,\overline z).
\end{equation}
Denote by $p\in U$ the point of coordinates $w_0$. Observe that fixing $w=w_0$, the diastasis $D_p(z):=D(z,w_0)$ is a K\"ahler potential for $g$ on $U$.

Calabi's criterion is the following:

\begin{theorem}[Calabi's criterion \cite{calabi}]\label{calabipi}
Let $(M,g)$ be a K\"ahler manifold. An open neighborhood of a point $p\in M$ admits a K\"ahler immersion into $\CP^N$ if and only if the $\infty\times\infty$ hermitian matrix of coefficients $(b_{jk})$ defined by:
\begin{equation}\label{expansioneD}
e^{D_p(z)}-1=\sum_{j,k=0}^\infty b_{jk}(z-p)^{m_j}(\overline z-\overline p)^{m_k}
\end{equation}
is positive semidefinite of rank at most $N$. 
\end{theorem}

Let now $(M,\omega_M)$ be a K\"ahler-Einstein manifold with Einstein constant $\lambda$, that is $\rho_M=\lambda\omega_M$. As described in Section \ref{sec: momentumconstruction}, the momentum construction gives a 1-parameter family of scalar flat K\"ahler metrics $\omega_{\varphi,\beta}$ on the polarized line bundle $(L,h)$ described by the K\"ahler potential:
\begin{equation}\label{kpotential}
\Psi(z,\xi)=\Phi(z)+4f\left(\frac{1}{2}\log[|\xi|^2 h(z)]\right),
\end{equation}
where we can take as $\Phi$ the diastasis function for $\omega_M$, centred at $z=0$.
We now describe the diastasis function for the metrics $g_{\varphi,\beta}$ and give a necessary condition for these metrics to be projectively induced, which follows directly by applying the Calabi's criterion to them. 

 By \eqref{diastasis}, the diastasis function associated to $\omega_{\varphi,\beta}$, centred at $p=(s,0)$ with $s\in\R^+$ is:
\begin{equation}\label{dias}
D(z,\xi)|_p=\Phi(z)+4f\left(\frac{1}{2}\log\left(|\xi|^2 h(z)\right)\right)+4f\left(\frac{1}{2}\log s^2\right)-4f\left(\frac{1}{2}\log(\xi s)\right)-4f\left(\frac{1}{2}\log(\overline\xi s)\right),
\end{equation}
where we set $h(0)=1$.

In particular, for the fibre metric we have:
\begin{equation}\label{diastfibre}
D_p(\xi)|_{\text{fibre}}=4f\left(\frac{1}{2}\log\left(|\xi|^2 \right)\right)+4f\left(\frac{1}{2}\log s^2\right)-4f\left(\frac{1}{2}\log(\xi s)\right)-4f\left(\frac{1}{2}\log(\overline\xi s)\right).
\end{equation}

\begin{proposition}\label{prop: comega}
    Let $c>0$ be a positive real number. Then the metric $c\,\omega_{\varphi,\beta}=\omega_{\hat{\varphi},\hat\beta}$, where $\hat{\varphi}$ is the profile function defined by \eqref{profile} with parameters $\hat{\beta}:=\beta/c$ and $\hat{\lambda}:=\lambda/c$.
\end{proposition}
\begin{proof}
Observe that 
\[
c \omega_{\varphi,\beta}=c\omega_M+2i\partial\overline\partial cf(t),
\]
and $c\omega_M$ is a K\"ahler--Einstein metric with Einstein constant $\frac{\lambda}{c}$. Conclusion follows since by Proposition \ref{cf} $\hat f=cf$ is a solution to $y''=\hat\varphi(y')$.
\end{proof}
\begin{remark}\rm
We note that if $g$ is a scalar flat projectively induced K\"ahler metric, then its (scalar flat) multiples $cg$ may not be so. If the base manifold is K\"ahler--Einstein with Einstein constant $\lambda$, we find a close connection between the parameters $c$ and $\lambda$ (as in the previous proposition). Namely, it turns out that varying the parameter $c$ over the positive real line, corresponds to construct the Hwang--Singer metrics on the same line bundle over a rescaled K\"ahler-Einstein manifold with Einstein constant $\frac{\lambda}{c}$. Thus it is equivalent to study the metric $c\omega_\varphi$ as $c$ varies and the metric $\omega_{\varphi}$ as $\lambda$ varies in the base manifold. 
\end{remark}
\begin{lemma}\label{condnec}
In the notation above, a necessary condition for the metric $\omega_{\varphi,\beta}$ to be projectively induced is that:
\begin{equation}\label{dercn}
n\left(\lambda+2\beta\right)\geq -4.
\end{equation}

\end{lemma}
\begin{proof}
By Calabi's Criterion Theorem \ref{calabipi}, since $\frac{\partial^4 (e^{D_p} -1)}{\partial\xi^2\overline\partial\xi^2}|_p$ is an element on the diagonal of the matrix $(b_{jk})$ in \eqref{expansioneD}, a necessary condition for the metric $\omega_{\varphi,\beta}$ to be projectively induced is that:
\begin{equation}\label{D2}
\frac{\partial^4 (e^{D_p} -1)}{\partial\xi^2\overline\partial\xi^2}|_p=\frac{1}{s^4}\left(\frac{1}{4} f^{(4)}\left(\frac{\log s^2}{2}\right)-f^{(3)}\left(\frac{\log s^2}{2}\right)+2 f''\left(\frac{\log s^2}{2}\right)^2+f''\left(\frac{\log s^2}{2}\right)\right)\geq 0.
\end{equation}
By \eqref{derivatef} and since $\varphi(\tau)>0$ for every $\tau\in\R^+$, \eqref{D2} is equivalent to:
\begin{equation}%\label{dercn}
4+8\varphi(\mu_0)-4\varphi'(\mu_0)+\varphi'(\mu_0)^2+\varphi(\mu_0)\varphi''(\mu_0)\geq 0,\nonumber
\end{equation}
where $\mu_0:=f'\left(\frac{\log s^2}{2}\right)$, i.e.:
\begin{equation}%\label{dercn}
(2-\varphi'(\mu_0))^2 +\varphi(\mu_0)\left(8+\varphi''(\mu_0)\right)\geq 0,\nonumber
\end{equation}
that is:
\begin{equation}\label{dercn0}
 \varphi''(\mu_0)\geq -8-\frac{(2-\varphi'(\mu_0))^2}{\varphi(\mu_0)}.
\end{equation}
By the definition of $\varphi$ \eqref{profile},
$$
\varphi'(\mu_0)=\frac{2\left((n+1)\beta+\lambda\left(1-(1-\beta\mu_0)^n\right)+\left((\beta^2+1)(n^2-1)+\lambda((1-\beta\mu_0)^n)\right)\mu_0\right)}{(n+1)\beta(1-\beta\mu_0)^{n+1}},
$$

$$
\varphi''(\mu_0)=\frac{2n}{(1-\beta\mu_0)^{n+2}}\left(\lambda+2\beta+(n-1)\beta(\beta+\lambda)\mu_0\right).
$$
Since we can choose $s>0$ arbitrarily small, then \eqref{dercn0} must hold for $\mu_0\to0$ (see Remark \ref{finfty}). It is not hard to see that as $\mu_0\rightarrow 0$,
$$
\frac{(2-\varphi'(\mu_0))^2}{\varphi(\mu_0)}\rightarrow 0,
$$
and
$$
\varphi''(\mu_0)\rightarrow2n\left(\lambda+2\beta\right).
$$
Thus \eqref{dercn0} implies 
$$
n\left(\lambda+2\beta\right)\geq -4,
$$
as wished.
\end{proof}

\begin{remark}\rm  \label{sharper nc}
    Considering the $j$-th derivatives $\frac{\partial^{2j} (e^{D_p} -1)}{\partial\xi^j\overline\partial\xi^j}|_p$, we get sharper necessary conditions for the metric $g_{\varphi,\beta}$ to be projectively induced. Although, such conditions will always depend on the choice of $\beta$ and $\lambda$. 
\end{remark}
\begin{remark}\rm  \label{ricciflatcase}
    When $\beta=-\lambda$, the metric $g_{\varphi,\beta}$ is Ricci-flat. In this case condition \eqref{dercn} gives that $g_{\varphi,\beta}$ is not projectively induced for any $\lambda> \frac4n$. This estimate can be improved to $\lambda\geq 1$ also for $n=2$, $3$, and $4$, by computing the $4$-th derivative $\frac{\partial^{8} (e^{D_p} -1)}{\partial\xi^4\overline\partial\xi^4}|_p$, evaluated at $\mu_0=\frac1{100\lambda}$. Further, observe that when $n=1$, $g_{\varphi,\beta}$ is the Eguchi--Hanson metric on $\mathds C{\rm P}^1$, which has been proven to be not projectively induced in \cite{loizeddazuddas2}. As before, observe that such condition can be improved considering higher derivatives but will always depend on the choice of $\lambda$, as in the above remark. 
\end{remark}

\section{Hwang--Singer metrics on $\C^{n+1}$ }\label{flatcase}
Let $(M,\omega_M)=(\mathds C^{n},\omega_0)$, where $\omega_0$ is the canonical flat metric, i.e. $\omega_0=\frac i2\partial\bar\partial||z||^2$. The momentum construction in this case gives a $1$-parameter family of scalar flat K\"ahler metrics $\omega_{\varphi,\beta}$ on $\mathds C^{n+1}$ described by the K\"ahler potential (see \ref{kpotential}):
\begin{equation}\label{potcn}
\Phi(z,\xi):=||z||^2+4f\left(\frac12\log\left[|\xi|^2e^{-\frac{\beta}2||z||^2}\right]\right),
\end{equation}
obtained setting $h(z)=e^{-\frac{\beta}{2}||z||^2}$ in \eqref{kpotential}, for $\beta<0$, so that $\gamma=-i\partial\bar \partial \log h(z)=\beta \frac{i}{2}\partial\bar \partial||z||^2=\beta\omega_M$
and by \eqref{dias} the diastasis function for $\omega_{\varphi,\beta}$ centred at $(z, \xi)=(s,0)$ reads:
\begin{equation}\label{diastcn}
D_{(s,0)}(z,\xi)=||z||^2+4f\left(\frac12\log\left[|\xi|^2e^{-\frac{\beta}2||z||^2}\right]\right)+4f\left(\frac{1}{2}\log s^2\right)-4f\left(\frac12\log\left[\xi s\right]\right)-4f\left(\frac12\log\left[{\bar\xi}s\right]\right).
\end{equation}

The profile function, obtained setting $\lambda=0$ in \eqref{profile} is given by:
\begin{equation}\label{phicn}
 \varphi(\tau)=\frac{2\tau}{(1-\beta\tau)^n},   
\end{equation}
for $\tau\in [0,+\infty)$.

In order to prove the first part of Theorem \ref{main1}, i.e. that $(\mathds C^{n+1},c\omega_{\varphi,\beta})$ is not projectively induced for any $c$ and $\beta$, let us first show how to drop the dependence on the parameters $\beta$ and $c$.
\begin{lemma}\label{betac}
    Up to an affine change of coordinates on $\mathds C^{n}$, the metric $c\omega_{\varphi,\beta}$ on $\mathds C^{n+1}$ is equivalent to $\omega_{\varphi,-\frac{1}{c}}$.
\end{lemma}
\begin{proof}
Let us first deal with $\beta$. The metric $\omega_{\varphi,-1}$ is obtained by a momentum construction on $(\C^{n},\omega_0)$ with profile $\varphi(\tau)=\frac{2\tau}{(1+\tau)^n}$. Perform a change of coordinates on $\C^{n}$ by setting $z'=\frac1{\sqrt{-\beta}} z$. Then:
$$
\omega_0=\frac i2\partial\bar\partial ||z||^2=-\beta\frac i2\partial\bar\partial ||z'||^2\;,
$$
Observe that while in the $z$ coordinates $\gamma=-\omega_0$, in the coordinates $z'$, $\gamma=\beta\omega_0$. The determinant of the endomorphism $\rm Id-\tau B$ (see Section \ref{sec: momentumconstruction} after formula \eqref{profile}), that in the $z$ coordinates was $(1+\tau)^n$, now in $z'$ reads $(1-\beta \tau)^n$. Thus the change of coordinates, transforms the metric $\omega_{\varphi,-1}$ on $\C^{n+1}$ in the metric $\omega_{\varphi,\beta}$.

Let us now prove that the multiplication of $\omega_{\varphi,-1}$ by $c>0$ is equivalent to consider $\omega_{\varphi,-\frac1c}$. 
By \eqref{potcn} a K\"ahler potential for $c\omega_{\varphi,-1}$ is given by $$
c\Phi(z,\xi)=c||z||^2+4cf\left(\frac12\log\left[|\xi|^2e^{\frac{1}2||z||^2}\right]\right).
$$
Performing a change of coordinates $z'=\sqrt{c}z$ we get
$$
c\Phi(z',\xi)=||z'||^2+4cf\left(\frac12\log\left[|\xi|^2e^{\frac{1}{2c}||z'||^2}\right]\right)\;.
$$
Conclusion follows observing that by Proposition \ref{cf}, if $f$ satisfies the ODE given by $\varphi(\tau)=\frac{2\tau}{(1+\tau)^n}$, then $cf$ satisfies the ODE given by $\varphi(\tau)=\frac{2\tau}{(1-\frac1c\tau)^n}$.
\end{proof}

\section{Hwang--Singer metrics on line bundles over $\CP^1$}\label{projectivecase}

 Let $L$ be a holomorphic line bundle over $\CP^1$ endowed with the Fubini-Study metric normalized so that $\lambda=1$, that is, in affine coordinates  $z=\frac{Z_1}{\frac{1}{2}Z_0}$ on $U_0=\{Z_0\neq0\}$
\[
\omega_{\text{FS}}=\frac{i}{2}\partial\overline\partial\; 4\log(1+\frac{1}{4}|z|^2).
\]
 Since $L$ is a holomorphic line bundle over $\CP^1$, then $L$ is of the form $\mathcal{O}(-k)$, $k\in\mathbb{Z}$.\\
The natural hermitian metric on the line bundle $\mathcal{O}(-1)$ on $\CP^1$ is given by restricting the hermitian metric of $\C^2$ to each fiber $l=L_x\subset\C^2$. So if $\{U_0,U_1\}$ is a cover of $\CP^1$ with
\[
U_\alpha=\{Z_\alpha\neq0\}\quad,\quad\alpha=0,1,
\]
 then
\[
h_\alpha=\frac{|Z_0|^2+|Z_1|^2}{|Z_\alpha|^2}\quad,\quad\alpha=0,1.
\]
So, on each open set $U_\alpha$, if we take $z$ as local coordinate, we have
\[
h(z)=1+\frac{1}{4}|z|^2.
\]
For $k>0$, the line bundles $\mathcal{O}(-k):=\mathcal{O}(k)^*=\mathcal{O}(1)\otimes\dots\otimes\mathcal{O}(1)$ inherit natural hermitian structures given by 
\[
h_k(z)=\left(1+\frac{1}{4}|z|^2\right)^k.
\]
The curvature form is then 
\[
\gamma=-i\partial\overline\partial\log\left(1+\frac{1}{4}|z|^2\right)^k=-\frac{k}{2}\omega_{\text{FS}} 
\]
and the line bundle is polarized, with $\lambda=1$ and $\beta=-\frac{k}{2}$, with $k$ a positive integer.\\
So the profile \ref{profile} reads 
\[
\varphi_k(\tau)=\frac{2\tau+\tau^2}{1+\frac{k}{2}\tau},
\]
and the momentum construction gives a 1-parameter family $\omega_k:=\omega_{\varphi,-\frac{k}{2}}$ of scalar flat K\"ahler metrics  on the polarized line bundle $\mathcal{O}(-k)$ described by the potentials
\begin{equation}\label{potenziale cp1}
\Psi(z,\xi)=4\log\left(1+\frac{1}{4}|z|^2\right)+4i\partial\overline\partial f\left(\frac{1}{2}\log\left[|\xi|^2 \left(1+\frac{1}{4}|z|^2\right)^k\right]\right).
\end{equation}

\begin{remark}
    For $k=0$, the metric $\omega_k$ reduces to the local product metric on $\CP^1\times\C$, see Remark \ref{prodmetrics}.
\end{remark}

In the proof of Theorem \ref{main2}, we need the following lemma.

\begin{lemma}\label{neccondcp1}
The metric $\omega_k$ on $\mathds C{\rm P}^1$ is not projectively induced for any $k\geq 3$.
\end{lemma}
\begin{proof}
Let $p\in \mathcal O(-k)$ be the point of coordinates $(s,0)$ and let ${D_p}$ be the diastasis function for the metric $\omega_k$ as in \eqref{diastfibre}.
The fourth derivative of $ e^ {D_p}-1$ evaluated at $p$ is given by:
\begin{equation*}
    \begin{split}
         \frac{\partial^8 (e^{D_p} -1)}{\partial\xi^4\overline\partial\xi^4}|_p  = \frac{1}{s^8}\bigg(
         &+24 {f''}^4+216 {f''}^3+f''' (3 f^{(5)}-45 f^{(4)}-66)+(18 f^{(4)}-216 f'''+242) {f''}^2+\\
         &+\frac{1}{64} (f^{(8)}-24 f^{(7)}+232 f^{(6)}-1152 f^{(5)}+136 (f^{(4)})^2+3088 f^{(4)})+\\
         &+(f^{(6)}-18 f^{(5)}+125 f^{(4)}+36 {f'''}^2-396 f'''+36) f''+114 {f'''}^2\bigg)\left(\frac{\log s^2}{2}\right),
    \end{split}
\end{equation*}
that written in terms of $\varphi(\mu_0)$ with $\mu_0=f'\left(\frac{\log s^2}{2}\right)$, up to the multiplication by the positive constant $\frac{1}{s^8}$, reads $\frac{1}{64} \phi (\mu_0 ) A(\varphi(\mu_0))$, with (to simplify the notation we drop the dependence from $\mu_0$ in $\varphi(\mu_0)$ and its derivatives):
\begin{equation} 
    \begin{split}
       A(\varphi(\mu_0))= & \varphi ^{(6)} \varphi ^5+((\varphi ')^3-12 (\varphi ')^2+44 \varphi '-48)^2+\varphi (\varphi ')-2) (-8 (193 \varphi ''+968)+\\
        &(\varphi ')^3 (57 \varphi ''+392)-2 (\varphi '))^2 (255 \varphi ''+1624)+4 \varphi ' (383 \varphi ''+2200))+2 \varphi ^2 (16 (-36 \varphi ^{(3)}\\
        &+29 (\varphi '')^2+250 \varphi ''+432)+61 \varphi ^{(3)} (\varphi ')^3+2 (\varphi ')^2 (96 (9-2 \varphi ^{(3)}+ 45 (\varphi '')^2+436 \varphi '')+\\
        &-4 \varphi ' (-203 \varphi ^{(3)}+102 (\varphi '')^2+936 \varphi ''+1728))+
        2 \varphi ^3 (17 (\varphi '')^3+196 (\varphi '')^2\\
        &+2 \varphi ^{(4)} (19 (\varphi ')^2-66 \varphi '+58)+64 \varphi ^{(3)} (5 \varphi '-9)+12 (\varphi ^{(3)} (8 \varphi '-15)+48) \varphi ''+768)+\\
        &\varphi ^4 (15 (\varphi ^{(3)})^2+
        8 \varphi ^{(5)} (2 \varphi '-3)+\varphi ^{(4)} (26 \varphi ''+64)),
    \end{split}\nonumber
\end{equation}
since $\varphi(\mu_0)$ is positive, the sign of $\frac{\partial^8 (e^{D_p} -1)}{\partial\xi^4\overline\partial\xi^4}|_p $ is the same as that of $A(\varphi(\mu_0))$.
From \eqref{profile}, we get:
$$
\varphi_k^{(j)}(\tau)=(-1)^{j+1}\frac{8j!(k-1)k^{j-2}}{(2+k\tau)^{j+1}},
$$
that substituted into the expression of $A(\varphi(\mu_0))$ gives:
$$
\frac{\mu_0^3}{2^{7}3(2 + k \mu_0)^{12}}P_k(\mu_0),
$$
where $P_k(\mu_0)$ is the polynomial in $\mu_0$:
$$
P_k(\mu_0)=105 - 113 k + 48 k^2 - 8 k^3+\sum_{s=1}^{12}q_s(k)\mu_0^s,
$$
for given $q_s(k)$ that are not relevant for our analysis.
 Since $\mu_0$ can be chosen small enough in $ [0,+\infty)$ taking $s\to0$, the sign of $A(\varphi(\mu_0))$ is the same as the sign of $P_k(\mu_0)$ for positive values of $\mu_0$. Conclusion follows by noticing that
$$
\lim_{\mu_0\rightarrow 0}P_k(\mu_0)=105 - 113 k + 48 k^2 - 8 k^3,
$$
and the right hand side is negative for any $k\geq 3$.
\end{proof}

\section{Asymptotic expansion of $\omega_{\varphi,\beta}$ and proofs of Theorem \ref{main1} and Theorem \ref{main2}}\label{sec: englisexp}
Throughout this section let us write $(X,\omega_{\varphi,\beta})$ for either $X=\mathds C^{n+1}$ or $X=\mathcal O(-k)$. Let ${\hat L}$ be a holomorphic line bundle over $X$ and let $h_{\hat L}$ be an hermitian metric on ${\hat L}$ such that ${\rm Ric}(\hat h)=\omega_{\varphi,\beta}$. Notice that such an $({\hat L},\hat h)$ exists if and only if $\omega_{\varphi,\beta}$ is integral. In our case, this occurs since the base metric $\omega_M$ is an integral form and:
\[
[\omega_{\varphi,\beta}]=[\pi^*\omega_M+2i\partial\overline\partial f(t)]=[\pi^*\omega_M],
\]
where $\pi\!:L\rightarrow M$ is the projection given in Section \ref{sec: momentumconstruction} (here $M=\mathds C^n$ or $\mathds C{\rm P}^1$). 
Consider the tensor power $(\hat L^\alpha,\hat h_\alpha)$ and let $\mathcal H_\alpha$ be the space of global holomorphic sections of $\hat L^\alpha$. 
In order to define the $\epsilon$-function for $(X,\omega_{\varphi,\beta})$ we first need to show that $\mathcal H_\alpha\neq\{0\}$.
\begin{lemma}\label{1inH}
In the notation above, $1\in \mathcal H_\alpha$ for either $X=(\mathds C^{n+1},\omega_{\varphi,\beta})$ or $X=(\mathcal O(-k),\omega_k)$).
\end{lemma}
\begin{proof}
Observe that by formula (2.21) in \cite{hwangsinger} we have:
$$
\frac{\omega_{\varphi,\beta}^{n+1}}{(n+1)!}=\varphi Q\det(g_M)\frac{1}{|\xi|^2}\left(\frac{i}{2}\right)^{n+1}d\xi\wedge d\bar \xi\prod_{j=1}^n dz_j\wedge d\bar z_j.
$$ 

Let us deal first with the case $X=\mathds C^{n+1}$. In this case, $\mathcal H_\alpha$ is the 
weighted Hilbert space of global holomorphic functions over $\C^{n+1}$ that are $L^2$ limited in norm, namely:
$$
\mathcal H_\alpha=\left\{u\in {\rm Hol}(\C^{n+1})|\int_{\C^{n+1}}|u|^2e^{-\alpha\Phi}\frac{\omega_{\varphi,\beta}^{n+1}}{(n+1)!}<+\infty\right\},
$$
where $\Phi$ is given by \eqref{potcn}.

Due to Lemma \ref{betac}, we can set $\beta=-1$. In order to prove that $1\in \mathcal H_\alpha$, it is enough to check the convergence of the integral:
\begin{equation}\label{intC}
    \begin{split}
\int_{\mathds C^{n+1}}&e^{-\alpha(||z||^2+4f(t)))}\frac{2f'(t)}{|\xi|^2}\left(\frac{i}{2}\right)^{n+1}d\xi\wedge d\bar \xi\prod_{j=1}^n dz_j\wedge d\bar z_j\\
=&\pi^{n+1}\int_0^\infty\cdots\int_0^\infty\int_0^\infty e^{-\alpha(\sum_jr_j+4f(\hat t)))}\frac{2f'(\hat t)}{r_0}dr_0\prod_{j=1}^n dr_j,
\end{split}
\end{equation}
where we set polar coordinates $\xi:=\rho_0e^{i\theta_0}$, $z_j:=\rho_je^{i\theta_j}$ and $r_j:=\rho_j^2$, $j=0,\dots, n$, and we denote $\hat t=\frac12(\log r_0+\sum_j r_j)$.
The function under the integral is positive and smooth, since $\frac{\varphi(t)}{|\xi|^2}\rightarrow g_{0\overline0}$ as $|\xi|^2\rightarrow 0$, so its integral converges inside any closed  ball of ray $R>0$ centered at the origin. Thus, it is enough to check that the integral outside the ball is finite. Using that the function under the integral is positive and that:
$$
e^{-\alpha(\sum_jr_j+4f(\hat t)))}\frac{2f'(\hat t)}{r_0}\leq e^{-\alpha(4f(\hat t)))}\frac{2f'(\hat t)}{r_0}=-\frac1\alpha\frac{d}{dr_0}e^{-4\alpha f(\hat t)},
$$
we have:
\begin{equation}\label{intC2}
    \begin{split}
\int_R^\infty\cdots\int_R^\infty\int_R^\infty e^{-\alpha(\sum_jr_j+4f(\hat t)))}\frac{2f'(\hat t)}{r_0}dr_0\prod_{j=1}^n dr_j
\leq&  -\frac{1}\alpha\int_R^\infty\cdots\int_R^\infty\int_R^\infty\frac{d}{dr_0}e^{-4\alpha f(\hat t)}dr_0\prod_{j=1}^n dr_j\\
=&  \frac{1}\alpha\int_R^\infty\cdots\int_R^\infty e^{-4\alpha f(\frac12\log R_0+\frac12\sum_j r_j)}\prod_{j=1}^n dr_j.
\end{split}
\end{equation}
The last integral in \eqref{intC2} converges at least for $\alpha>n/4$ and for a large enough $R$, since:
$$
 e^{-4\alpha f(\frac12\log R_0+\frac12\sum_j r_j))}\leq \frac{1}{r_1^{4\alpha/n}\cdots r_n^{4\alpha/n}}.
$$
More precisely, since $f$ is an increasing function and $\frac{\partial^2}{\partial r_j^2}f=\frac14f''>0$, there exists $R\in \mathds R$ such that for $r_j>R$, $j=1,\dots, n$,
$$
f\left(\frac12\log R_0+\frac12\sum_j r_j\right)\geq f(r_j)\geq \log r_j,
$$ 
 thus
$$
f\left(\frac12\log R_0+\frac12\sum_j r_j\right)\geq \frac1n\sum_{j=1}^n \log r_j.
$$

Let us now deal with $X=\mathcal O(-k)$ over $\mathds C{\rm P}^1$. In this case it is enough to check the convergence of the following integral over the chart $\mathcal U_0\times \mathds C\simeq \mathds C^2$:
\begin{equation}
    \begin{split}
\int_{\mathds C^{2}}&\frac{e^{-4\alpha f(t)}}{(1+\frac14|z|^2)^{4\alpha+2}}\frac{2f'(t)+f'(t)^2}{|\xi|^2}\left(\frac i2\right)^2d\xi\wedge d\bar \xi\wedge dz\wedge d\bar z\\
=&\pi^{2}\int_0^\infty\int_0^\infty \frac{e^{-4\alpha f(\hat t)}}{(1+\frac14r_1)^{4\alpha+2}}\frac{2f'(\hat t)+f'(\hat t)^2}{r_0}dr_0 dr_1,
\end{split}\nonumber
\end{equation}
where we set polar coordinates $\xi=\rho_0e^{i\theta_0}$, $z_1=\rho_1e^{i\theta_1}$, $r_j:=\rho_j^2$, $j=0$, $1$, and set $\hat t:=\frac12\log r_0+\frac k2\log(1+\frac14r_1)$.
As before, since the function we are integrating is smooth on any closed ball of ray $R>0$ (since $\frac{\varphi(t)}{|\xi|^2}\rightarrow g_{0\overline0}$ as $|\xi|^2\rightarrow 0$), we reduce to check that the integral converges outside such ball. First observe that
\[
I_1:=\int_R^\infty  -e^{-4\alpha f(\hat t(r_0))}\frac{f'(\hat t(r_0))}{r_0}dr_0=\frac{1}{2\alpha}\int_R^\infty \frac{d}{dr_0}e^{-4\alpha f}=\frac{1}{2\alpha}\left[e^{-4\alpha f}\right]_R^\infty<\infty
\]
and
\begin{equation}
\begin{split}
I_2:\!&=\int_R^\infty\int_R^\infty \frac{e^{-4\alpha f(\hat t)}}{(1+\frac14r_1)^{2}}\frac{2f'(\hat t)}{r_0}dr_0 dr_1=-\frac{1}{\alpha}\int_R^\infty\int_R^\infty \frac{1}{(1+\frac{1}{4}r_1)^2}\frac{d}{dr_0}e^{-4\alpha f}dr_0dr_1\\
&=-\frac{1}{\alpha}\int_R^\infty\frac{1}{(1+\frac{1}{4}r_1)^2} \left[e^{-4\alpha f}\right]_R^\infty dr_1=\frac{1}{\alpha}\int_R^\infty\frac{e^{-4\alpha f(\hat t(r_1))}}{(1+\frac{1}{4}r_1)^2}dr_1<\infty.
\end{split}\nonumber
\end{equation}
So, since $\alpha>0$, we have
\begin{equation}
    \begin{split}
\int_R^\infty\int_R^\infty \frac{e^{-4\alpha f(\hat t)}}{(1+\frac14r_1)^{4\alpha+2}}\frac{2f'(\hat t)+f'(\hat t)^2}{r_0}dr_0 dr_1=&\int_R^\infty\int_R^\infty \frac{e^{-4\alpha f(\hat t)}}{(1+\frac14r_1)^{4\alpha+2}}\frac{2f'(\hat t)}{r_0}dr_0 dr_1+\\
&+\int_R^\infty\int_R^\infty\frac{e^{-4\alpha f(\hat t)}}{(1+\frac14r_1)^{4\alpha+2}}\frac{f'(\hat t)^2}{r_0}dr_0 dr_1\\
\leq&\ I_2+\int_R^\infty\int_R^\infty\frac{e^{-4\alpha f(\hat t)}}{(1+\frac14r_1)^{2}}\frac{f'(\hat t)^2}{r_0}dr_0 dr_1.
\end{split}\nonumber
\end{equation}
It remains to check that:
\[
I:=\int_R^\infty\int_R^\infty\frac{e^{-4\alpha f(\hat t)}}{(1+\frac14r_1)^{2}}\frac{f'(\hat t)^2}{r_0}dr_0 dr_1
\]
converges. Integrating by parts, since:
\[
\frac{e^{-4\alpha f(\hat t)}}{(1+\frac14r_1)^{2}}\frac{f'(\hat t)^2}{r_0}=-\frac{2}{k\alpha}\frac{d}{dr_1}e^{-4\alpha f(\hat t)}\frac{f'(\hat t)}{r_0(1+\frac{1}{4}r_1)}.
\]
we get: 
\begin{equation}\label{integrali}
    \begin{split}
        I=&-\frac{2}{k\alpha}\int_R^\infty\int_R^\infty 
\frac{d}{dr_1}e^{-4\alpha f(\hat t)}\frac{f'(\hat t)}{r_0(1+\frac{1}{4}r_1)}dr_0 dr_1\\
        =&-\frac{2}{k\alpha} \left\{ \int_R^\infty \left[ e^{-4\alpha f}\frac{f'}{r_0(1+\frac{1}{4}r_1)}\bigg|_R^\infty-\frac{1}{r_0}\int_R^\infty e^{-4\alpha f}\frac{f'' \frac{k}{8}-\frac{f'}{4}}{(1+\frac{1}{4}r_1)^2}dr_1\right]dr_0\right\}\\
        =&-\frac{2}{k\alpha}\left\{\frac{I_1}{(1+\frac{1}{4}R)} -\frac k8\int_R^\infty\int_R^\infty \frac{e^{-4\alpha f}f'' }{r_0(1+\frac{1}{4}r_1)^2}dr_1dr_0+\frac18\int_R^\infty\int_R^\infty \frac{2f'e^{-4\alpha f}}{r_0(1+\frac{1}{4}r_1)^2}dr_1dr_0\right\}\\
        =&-\frac{2}{k\alpha}\left\{\frac{I_1}{(1+\frac{1}{4}R)}-\frac k8\int_R^\infty\int_R^\infty  \frac{e^{-4\alpha f}(2f'+(f')^2)}{r_0(1+\frac{1}{4}r_1)^2(1+\frac{k}{2}f')}dr_1dr_0+\frac{I_2}{8}\right\}\\
        \leq&-\frac{2}{k\alpha}\left\{\frac{I_1}{(1+\frac{1}{4}R)}-\frac k8\int_R^\infty\int_R^\infty  \frac{e^{-4\alpha f}(2f'+(f')^2)}{r_0(1+\frac{1}{4}r_1)^2}dr_1dr_0+\frac{I_2}{8}\right\}\\
        =&-\frac{2}{k\alpha}\left\{\frac{I_1}{(1+\frac{1}{4}R)}-\frac{k}{8}I_2-\frac{k}{8}I+\frac{I_2}{8}\right\}
    \end{split}
\end{equation}
where in the second equality we used that
$$
\lim_{r_1\rightarrow +\infty}\frac{f'e^{-4\alpha f}}{(1+\frac{1}{4}r_1)}=0
$$
as follows by applying de l'Hopital and using that $f''=\frac{2f'+f'^2}{1+\frac k2 f'}$. Further the inequality follows by $(1+\frac{k}{2}f')>1$, since $f'$ is a positive function.
From \eqref{integrali} we obtain:
\[
\left(1-\frac{1}{4\alpha}\right)I\le C
\]
for a suitable constant $C\in\R$. In particular $I$ converges at least for $\alpha>\frac14$.
\end{proof}

Unlike the compact case, for a noncompact manifold, it is not guaranteed in general the existence of the Engli\v{s} expansion of the function $\epsilon_{kg}$ and only partial results in this direction are known (see e.g. \cite{englis} for the case of strongly pseudoconvex bounded domains in $\C^n$ with real analytic boundary). In \cite[Theorem 6.1.1]{mamarinescu} X. Ma and G. Marinescu state sufficient conditions for the expansion to exist in a very general context. A version of their theorem adapted to our setting reads as follows (cf. \cite[Theorem 7]{loizeddazuddas}):
\begin{theorem}\label{th: mamarinescu}
Let $(X,g,\omega)$ be a complete K\"ahler manifold and let $(\hat L,\hat h)$ be an hermitian line bundle on $X$. Then $\epsilon_{\alpha g}$ admits an asymptotic expansion in $\alpha$ with coefficients given by \ref{coeffespan} provided there exist constants $l>0$ and $c>0$ such that
\begin{equation}\label{condma}
iR^{\hat L} > l\;\omega\quad,\quad iR^{\det}>-c\;\omega\quad,\quad |\partial\omega|_g<c\;,
\end{equation}
where $R^{\det}$ denotes the curvature of the connection on $\det(T^{1,0}X)$ induced by $g$ and $R^{\hat L}$ the curvature of the connection on $\hat L$ induced by the hermitian metric $\hat h$. 
\end{theorem}
In the following theorem we prove that conditions \eqref{condma} hold for Hwang--Singer metrics based on a K\"ahler--Einstein polarised manifold.  
Recall that from \cite{hwangsinger} Section 2, the Ricci form $\rho_\varphi$ of $\omega_{\varphi,\beta}$ is given by
$$
\rho_\varphi=\pi^*\rho_M+\frac{1}{2Q}(\varphi Q)'(\tau)\pi^*\gamma-\frac{1}{2\varphi}[\frac{1}{Q}(\varphi Q)']' d\tau\wedge d^c\tau,
$$
thus, when $\omega_M$ is polarized we have:
\begin{equation}\label{ricciform}
\begin{split}
\rho_\varphi
&=\pi^*(\lambda\omega_M)+\frac{1}{2Q}(\varphi Q)'(\tau)\pi^*(\beta\omega_M)-\frac{1}{2\varphi}[\frac{1}{Q}(\varphi Q)']' d\tau\wedge d^c\tau\\
&=\big(\lambda+\frac{\beta}{2Q}(\varphi Q)'\big)\pi^*\omega_M-\frac{1}{2\varphi}\big[\frac{1}{Q}(\varphi Q)'\big]' d\tau\wedge d^c\tau.
\end{split}
\end{equation}

\begin{theorem}\label{esistenzaepsilon}
    Let $\omega_{\varphi,\beta}$ be the the Hwang--Singer metric on a polarized line bundle over a K\"ahler--Einstein manifold with integral K\"ahler form. Then, if $\mathcal{H}_\alpha\neq\{0\}$, the Engli\v{s} expansion of the function $\epsilon_{\alpha g_{\varphi,\beta}}$ exists and the coefficients $a_j$ are given by \ref{coeffespan}. 
\end{theorem}
\begin{proof}
   Let us check that conditions \eqref{condma} hold for $\omega_{\varphi,\beta}$. 
    The first condition is satisfied for $l\in (0,1)$ since
    \[
    iR^{\hat L}=-i\partial\overline\partial\log \hat h=-i\partial\overline\partial\log e^{-\frac12\Psi}=\omega_{\varphi,\beta}\;,
    \]
    while the third condition is satisfied for every positive $c>0$, since $\partial\omega_{\varphi,\beta}=0$, being the metric K\"ahler. Let us now deal with the second condition. We want to show that there exists a positive $c>0$ such that the form given by
    \[    iR^{\det}+c\omega_{\varphi,\beta}=\rho_\varphi+c\omega_{\varphi,\beta}
    \]
    is positive. Thus, using \ref{ricciform} and \ref{omegaphibeta}, it is sufficient to show that there exists $c>0$ such that
   $$
        \left(\lambda+\frac{\beta}{2Q}(\varphi Q)'+c(1-\tau\beta)\right)\omega_M+\left(-\frac{1}{2\varphi}\big[\frac{1}{Q}(\varphi Q)'\big]' +\frac{c}{\varphi}\right)d\tau\wedge d^c\tau> 0.
   $$
    Being $\lambda\ge0$ and $\varphi>0$, we show that there exists $c>0$ such that
    \[
    \begin{cases}
        \frac{\beta}{2Q}(\varphi Q)'+c(1-\tau\beta)>0\\
        -\frac{1}{2}\big[\frac{1}{Q}(\varphi Q)'\big]' +c\ge0\;,
    \end{cases}
    \]
    namely, we want a positive $c$ that satisfies 
     \[
    \begin{cases}
        c>-\frac{\beta}{2Q}(\varphi Q)'\frac{1}{1-\tau\beta}\\
        c\ge2\big(\frac{1}{Q}(\varphi Q)'\big)'\;.
    \end{cases}
    \]
    Since $\frac{1}{1-\tau\beta}\le1$, we reduce to prove that 
    \[
    \frac{(\varphi Q)'}{Q}\quad,\quad \bigg(\frac{(\varphi Q)'}{Q}\bigg)'
    \]
    are limited functions, proving the existence of such a $c$. Using the expression of the profile function \ref{profile} and since $Q(\tau)=(1-\beta\tau)^n$, we get
    \begin{align*}
     \frac{(\varphi Q)'}{Q}&%\frac{2 (1-\beta \tau)^{-n} \left(-\lambda (1-\beta \tau)^n+\beta+\lambda\right)}{\beta}\\
     =\frac{2 \lambda}{\beta (1-\beta \tau)^{n}}-\frac{2 \lambda}{\beta}+\frac{2} {(1-\beta \tau)^{n}}\\
     &<-\frac{2 \lambda}{\beta}+\frac{2} {(1-\beta \tau)^{n}}\\
     &<-\frac{2 \lambda}{\beta}+2,
    \end{align*}
    and
    \begin{align*}
        \bigg(\frac{(\varphi Q)'}{Q}\bigg)'&=\frac{2 n (\beta+\lambda)} {(1-\beta \tau)^{n+1}}\leq 2 n (\beta+\lambda),
    \end{align*}
  concluding the proof.
\end{proof}

From the existence of an asymptotic expansion of the $\epsilon$-function it follows that the metric can be approximated by a sequence of projectively induced ones in the following way (cf. \cite[Corollary 9]{loizeddazuddas}).
\begin{lemma}\label{convergence}
Let $(M,g)$ be a polarized K\"ahler manifold such that the $1\in\mathcal H$, where $\mathcal H$ is the weighted Hilbert space of holomorphic functions on $M$ limited in norm. Then the $\epsilon$-function associated to $g$ exists and, if it admits an asymptotic expansion whose coefficients are given by \eqref{coeffespan}, then $g$ can be approximated by a sequence of projectively induced K\"ahler metrics.
\end{lemma}
\begin{proof}
Denote by $\omega$ the K\"ahler form associated to $g$. Let $F_\alpha:M\to \CP^{d_\alpha}$ be the coherent states map, i.e. $F_\alpha(x)=[\sigma_0(x):\dots:\sigma_j(x):\dots]$, where $\{\sigma_j\}_{j=0,1,\dots}$ is an orthonormal basis of $\mathcal H$ such that $\sigma_0\equiv 1$. Since $\mathcal H\neq\{0\}$, we can define the $\epsilon$-function for $g$ by \eqref{epsilondef}, and we have:
\[
F_\alpha^*\omega_{FS}=\alpha\omega +\frac{i}{2}\partial\overline{\partial}\log\epsilon_{\alpha g}.
\]
By \eqref{asympt}, since $a_0=1$, we have that $\lim_{\alpha\to\infty}\frac{1}{\alpha}F^*_\alpha \;g_{FS}=g$.    
\end{proof}
\begin{remark}\rm
    Observe that the assumption $1\in \mathcal H$ is needed to define the coherent states map. When $M$ is a compact polarized K\"ahler manifold, the existence of $F_\alpha$ is guaranteed by Kodaira's Theorem. In the noncompact case one can always define the map $F_\alpha$ for example when $g$ is regular, i.e. when $\epsilon_{\alpha g}$ is constant. In this case \eqref{epsilondef} implies that for each $x\in M$ there exists a nonvanishing $\sigma_j(x)$.
\end{remark}

We are now in the position of proving Theorem \ref{main1}.
\begin{proof}[Proof of Theorem \ref{main1}]
By Lemma \ref{betac} we can reduce ourselves to prove that $\omega_{\varphi,\beta}$ is not projectively induced for a given value of $\beta$. 
By Lemma \ref{condnec}, a necessary condition for the metric $\omega_{\varphi,\beta}$ on $\C^{n+1}$ to be projectively induced is that $\beta n\ge -2$. Thus, it is enough to set $\beta< -\frac{2}{n}$.

The second part follows by Lemma \ref{1inH}, Theorem \ref{esistenzaepsilon} and Lemma \ref{convergence}.
\end{proof}
Let us now complete the proof of Theorem \ref{main2}.

\begin{proof}[Proof of Theorem \ref{main2}]
    By Lemma \ref{neccondcp1}, $\omega_k$ is not projectively induced for any $k\geq3$. For $k=2$, $\omega_2$ is the Eguchi-Hanson metric on the canonical line bundle $\mathcal{O}(-2)$, that is not projectively induced as shown by A. Loi, M. Zedda, F. Zuddas in \cite{loizeddazuddas2}. For $k=1$, $\omega_1$ is the Burns-Simanca metric that is projectiveley induced as shown by F. Cannas Aghedu and A. Loi in \cite{agheduloi}. The second part follows by Lemma \ref{1inH}, Theorem \ref{esistenzaepsilon} and Lemma \ref{convergence}. Finally, a direct computation (see Appendix \ref{appendix} below), gives:
    $$
a_2=-\frac{48 (k-1) \left(k^2 \tau-2 k \tau-2\right)}{(k \tau+2)^6},
$$
that is identically zero if and only if $k=1$, concluding the proof. 
\end{proof}

\begin{remark}\rm
In \cite{agheduloi}, F. Cannas Aghedu and A. Loi showed that the Simanca metric $g_1$ is projectively induced, and this implies that any of its integer multiples $kg_1$ also are. We note here that these are the only possible multiples that can be K\"ahler immersed in $\CP^\infty$. In fact, by momentum construction, the Simanca metric on $\mathcal{O}(-1)$ arises as a metric on a line bundle over $\CP^1$. In particular, $\CP^1$ is a K\"ahler submanifold of $\mathcal O(-1)$ (obtained setting the fibre coordinate $\xi=0$) and the Fubini--Study form is not integral when multiplied by a noninteger factor.
\end{remark}

\appendix
\section{Computations of $a_2$} \label{appendix}
We compute here the $a_2$ coefficients for the metrics $\omega_{\varphi,\beta}$ in the case where the base manifold $M$ is the complex projective line $\CP^1$, completing the proofs of Theorem \ref{main2}.
From \eqref{potenziale cp1}, the metric $g_k$ reads:
\[
g_k=\begin{pmatrix}
    \frac{k^2 |z|^2 f''(t)+8 k f'(t)+16}{\left(|z|^2+4\right){}^2} &  \frac{k z f''(t)}{\xi \left(|z|^2+4\right)}\\
    \frac{k \overline z f''(t)}{\overline\xi \left(|z|^2+4\right)} & 
    \frac{f''(t)}{|\xi|^2}
\end{pmatrix}.
\]
It follows that
\[
\det(g_k)=\frac{(1+\frac{k}{2}\tau)\varphi(\tau)}{|\xi|^2(1+\frac{1}{4}|z|^2)},
\]
and
\[
g_k^{-1}=\begin{pmatrix}
    \frac{\left(|z|^2+4\right){}^2}{8 \left(k f'(t)+2\right)} & -\frac{k \overline\xi z \left(|z|^2+4\right)}{8 \left(k f'(t)+2\right)}\\
    -\frac{k \overline z \xi \left(|z|^2+4\right)}{8 \left(k f'(t)+2\right)} & \frac{|\xi|^2 \left(k^2 |z|^2 f''(t)+8 k f'(t)+16\right)}{8 f''(t) \left(k f'(t)+2\right)}
\end{pmatrix}.
\]
The norms of the Riemann and Ricci tensors are
\begin{equation*}
\begin{split}
    |R|^2 =\frac{1}{16} &\left(\frac{f^{(4)}(t)^2}{f''(t)^4}+\frac{f^{(3)}(t)^4}{f''(t)^6}-\frac{8 (k^3 f^{(3)}(t)-2 k f'(t)-4)}{(k f'(t)+2)^3}+\frac{8 k^4 f''(t)^2}{(k f'(t)+2)^4}\right.\\
    &\quad\left.-\frac{16 k^2 f''(t)}{(k f'(t)+2)^3}-\frac{2 f^{(3)}(t)^2 f^{(4)}(t)}{f''(t)^5}+\frac{4 k^2 f^{(3)}(t)^2}{f''(t)^2 (k f'(t)+2)^2}\right),
\end{split}
\end{equation*}
and
\begin{equation*}
    \begin{split}
        |\rm{Ric}|^2=\frac{1}{16} &\left(\frac{16}{(k f'(t)+2)^2}+\frac{f^{(3)}(t)^4}{f''(t)^6}+\frac{2 k^4 f''(t)^2}{(k f'(t)+2)^4}-\frac{8 k^2 f''(t)}{(k f'(t)+2)^3}-\frac{2 f^{(3)}(t)^2 f^{(4)}(t)}{f''(t)^5}\right.\\
        &\quad +\frac{4 k^2 f^{(3)}(t)^2}{f''(t)^2 (k f'(t)+2)^2}+\frac{2 k f^{(3)}(t) f^{(4)}(t)}{f''(t)^3 (k f'(t)+2)}-\frac{2 k (k f^{(4)}(t)+4 f^{(3)}(t))}{f''(t) (k f'(t)+2)^2}+\frac{f^{(4)}(t)^2}{f''(t)^4}+\\
        &\quad\left.-\frac{2 k f^{(3)}(t)^3}{(k f'(t)+2)f''(t)^4}\right).
    \end{split}
\end{equation*}
By \eqref{derivatef} with $\varphi(\tau)=\frac{2\tau+\tau^2}{1+\frac{k}{2}\tau}$,  the $a_2$ coefficient for the metrics $\omega_k$ on $\mathcal{O}(-k)$ is given by:
\begin{equation}\label{a2cp1}
a_2=-\frac{48 (k-1) \left(k^2 \tau-2 k \tau-2\right)}{(k \tau+2)^6}.
\end{equation}

\begin{remark}\rm
A similar computation for the Hwang--Singer metric on $\mathds C^{n+1}$ gives:
\begin{equation}
\begin{split}
    a_2(0,1)&=\frac{\beta^2}{4 (1-\beta\tau)^{2(n+2)}} \big(\beta^2 n^4 \tau^2+n (\beta^2 2^n \tau^2+2 \beta (2^n+4) \tau+2^n-4)+2^n (1-\beta^2 \tau^2)+\\
    &\quad+\beta n^3 \tau (\beta (2^n-2) \tau+4)+\beta n^2 \tau (\beta (2^n-3) \tau+2 \left(2^n+2\right))\big).
\end{split}\nonumber
\end{equation} 
\end{remark}

\end{document}